\documentclass[runningheads]{llncs}
\usepackage[T1]{fontenc}
\usepackage[utf8]{inputenc}
\usepackage[T1]{fontenc}
\usepackage[english]{babel}
\usepackage{amsmath,amssymb,amsfonts,siunitx,commath}
\usepackage{tikz,url}
\usepackage{geometry}
\usepackage{mathtools}
\mathtoolsset{showonlyrefs}
\usepackage{subcaption}
\usepackage{algorithm}
\usepackage{algpseudocode}
\usepackage{enumerate}
\usepackage{listings}
\usepackage{hyperref}

\newcommand{\R}{\mathbb{R}}
\newcommand{\N}{\mathbb{N}}
\renewcommand{\d}{\mathrm{d}}

\renewcommand{\P}{\mathcal{P}}
\renewcommand{\S}{\mathcal{S}}

\newcommand{\T}{\mathrm{T}}

\newcommand{\M}{\mathcal M}

\newcommand{\F}{\mathcal F}

\newcommand{\V}{\mathcal V}

\mathtoolsset{showonlyrefs}

\DeclareMathOperator{\dom}{dom}
\DeclareMathOperator{\Id}{Id}

\newcommand{\opt}{\mathrm{opt}}

\begin{document}

\title{Wasserstein Gradient Flows of the Discrepancy\\ 
  with Distance Kernel on the Line\thanks{Supported by
    the German Research Foundation (DFG)
    [grant numbers STE571/14-1, STE 571/16-1]
    and
    the Federal Ministry of Education and Research (BMBF, Germany)
    [grant number 13N15754].
  }}

\author{Johannes Hertrich \and Robert Beinert \and Manuel Gr\"af \and
  Gabriele Steidl}
\authorrunning{J. Hertrich et al.}
\institute{TU Berlin, Institute of Mathematics, Stra{\ss}e des 17. Juni 136, 10623 Berlin, Germany
  \email{\{hertrich,beinert, graef,steidl\}@math.tu-berlin.de}\\
  \url{https://tu.berlin/imageanalysis/}}

\maketitle              

\begin{abstract}
This paper provides results on
Wasserstein gradient flows between measures on the real line.
Utilizing the isometric embedding 
of the Wasserstein space $\P_2(\R)$ 
into the Hilbert space $L_2((0,1))$, 
Wasserstein gradient flows of functionals on 
$\P_2(\R)$ can be characterized as subgradient flows of associated functionals
on $L_2((0,1))$.
For the maximum mean discrepancy functional
$\mathcal F_\nu \coloneqq \mathcal D^2_{K}(\cdot, \nu)$
with the non-smooth negative distance kernel $K(x,y) = -|x-y|$,
we deduce a formula for the associated functional.
This functional appears to be convex, and we show that
$\mathcal F_\nu$ is convex along (generalized) geodesics. 
For the Dirac measure $\nu = \delta_q$, $q \in \R$ as end point of the flow,
this enables us to determine the Wasserstein gradient flows analytically. 
Various examples of Wasserstein gradient flows are given for illustration.
  
  \keywords{
	Maximum Mean Discrepancy 
	\and 
	Wasserstein gradient flows \and Riesz kernel.}
\end{abstract}

\section{Introduction}
Gradient flows provide a powerful tool for computing the minimizers of modeling functionals in certain applications.
In particular, gradient flows on the Wasserstein space  
are an interesting field of research that combines optimization with (stochastic) dynamical systems 
and differential geometry.
For a good overview on the theory, we refer to the books of Ambrosio, Gigli and Savar\'e \cite{BookAmGiSa05},
and Santambrogio \cite{S2015}. 
Besides Wasserstein gradient flows of the Kullback--Leibler (KL) functional $\text{KL}(\cdot,\nu)$
and the associated Fokker--Planck equation related to the overdamped Langevin dynamics, which
were extensively examined in the literature, see, e.g., \cite{JKO1998,Ot01,Pav2014},
flows of maximum mean discrepancy (MMD) functionals $\mathcal F_\nu \coloneqq \mathcal D^2_{K}(\cdot, \nu)$
became popular in machine learning \cite{ArKoSaGr19} and image processing \cite{EGNS2021}.
On the other hand, MMDs were used as loss functions in generative adversarial networks \cite{BSAG2018,DRG2015,LCCYP2017}.
Wasserstein gradient flows of MMDs 
are not restricted to absolutely continuous measures and have a rich structure depending on the kernel.
So the authors of \cite{ArKoSaGr19} showed that for smooth kernels $K$, 
particle flows are indeed Wasserstein gradient flows meaning that
Wasserstein flows starting at an empirical measure remain empirical measures and coincide with usual
gradient descent flows in $\R^d$. 
The situation changes for non-smooth kernels like the negative distance, where 
empirical measures can become absolutely continuous ones and conversely, i.e. particles may explode.
The concrete behavior of the flow depends also on the dimension, see \cite{CaHu17,ChSaWo22b,GCO2021,HGBS2022}.
The crucial part is the treatment of the so-called interaction energy within the discrepancy, which is repulsive and responsible for the
proper spread of the measure. This nicely links to another field of mathematics, namely potential theory \cite{BookLa72,BookSaTo97}.

In this paper, we are just concerned with Wasserstein gradient flows on the real line.
Optimal transport techniques that reduce the original transport to those on the
line were successfully used in several applications 
\cite{AABC2017,BBG2022,bonneel:hal-00881872,CCST2021,KPR2016,PKKR2017}.
When working on $\R$, we can exploit quantile functions of measures to embed the Wasserstein space $\P_2(\R)$ 
into the Hilbert space of (equivalence classes) of square integrable functions $L_2((0,1))$. 
Then, instead of dealing with functionals on $\P_2(\R)$,
we can just work with associated functionals which are uniquely defined on a cone of $L_2((0,1))$.
If the associated functional is convex, we will see that the original one is convex along 
(generalized) geodesics, which is a crucial property for
the uniqueness of the Wasserstein gradient flow.
Furthermore, we can characterize Wasserstein gradient flows
using regular subdifferentials in $L_2( (0,1) )$.
Note that the special case of Wasserstein gradient flows of the interaction energy 
was already considered in \cite{BoCaFrPe15}.
We will have a special look at the Wasserstein gradient flow of 
$\mathcal F_{\delta_q} \coloneqq \mathcal D^2_{K}(\cdot, \delta_q)$
for the negative distance kernel, i.e. flows ending in $\delta_q$.
We will deduce an analytic formula for this flow and provide several examples
to illustrate its behavior.
\paragraph{Outline of the paper.} 
In Section \ref{sec:dD}, we recall the basic notation on Wasserstein gradient flows in $d$ dimensions.
Then, in Section \ref{sec:1D}, we show how these flows can be simpler treated as
gradient descent flows of an associated function on the Hilbert space $L_2((0,1))$.
MMDs are introduced in Section \ref{sec:discr}.
Then, in Section \ref{sec:discr_1d}, we restrict our attention again to the real line
and show how the associated functional
looks for the MMD with negative distance kernel.
In particular, this functional is convex.
For the Dirac measure $\nu = \delta_q$, $q \in \R$,
we give an explicit formula for the Wasserstein gradient flow of the MMD functional.
Examples illustrating the behavior of the Wasserstein flows are
provided in Section \ref{sec:ex}.
Finally, conclusions are drawn in Section \ref{sec:conc}.

\section{Wasserstein Gradient Flows}\label{sec:dD}
Let $\mathcal M(\R^d)$ denote the space of $\sigma$-additive, signed measures and 
$\mathcal P(\R^d)$ the set of probability measures.
For $\mu \in \M(\R^d)$ and measurable $T\colon\R^d \to \R^n$,
the \emph{push-forward} of $\mu$ via $T$ is given by
$T_{\#}\mu \coloneqq \mu \circ T^{-1}$.
We consider the \emph{Wasserstein space}
$
  \P_2(\R^d) 
  \coloneqq 
  \{ \mu \in \P(\R^d) \colon \int_{\R^d}\|x\|_2^2 \, \d \mu(x) < \infty \}
$
equipped with the \emph{Wasserstein distance} 
$W_2\colon\P_2(\R^d) \times \P_2(\R^d) \to [0,\infty)$,
\begin{equation}        \label{def:W_2}
  W_2^2(\mu, \nu)
  \coloneqq 
  \min_{\pi \in \Gamma(\mu, \nu)} 
  \int_{\R^d\times \R^d}
  \|x -  y\|_2^2
  \, \d \pi(x, y),
  \qquad \mu,\nu \in \P_2(\R^d),
\end{equation}
where 
$
\Gamma(\mu, \nu)
\coloneqq
\{ \pi \in \P_2(\R^d \times \R^d):
(\pi_{1})_{\#} \pi = \mu,\; (\pi_{2})_{\#} \pi = \nu\}
$ 
and $\pi_i(x) \coloneqq x_i$, $i = 1,2$ for $x = (x_1,x_2)$. The set of optimal transport plans $\pi$ 
realizing the minimum in \eqref{def:W_2} is denoted by $\Gamma^{\rm{opt}}(\mu, \nu)$.
A curve $\gamma \colon I \to \P_2(\R^d)$ on an interval $I \subset \R$, is called a \emph{geodesic} 
if there exists a constant $C \ge 0$ 
such that 
\begin{equation}
    \label{eq:geodesic}
    W_2(\gamma(t_1), \gamma(t_2)) = C |t_2 - t_1|, \qquad \text{for all } t_1, t_2 \in I.
\end{equation}
The Wasserstein space is a geodesic space, meaning that any two measures $\mu, \nu \in \P_2(\R^d)$ can be connected by a geodesic.
The \emph{regular tangent space} at $\mu \in \P_2(\R^d)$ is given by
\begin{align} \label{tan_reg}
    \T_{\mu}\mathcal P_2(\R^d)
    &\coloneqq 
      \overline{
      \left\{ \lambda (T- \Id): (\Id ,T)_{\#} \mu \in \Gamma^{\opt} (\mu , T_{\#} \mu), \; \lambda >0
      \right\} }^{L_{2,\mu}}.
\end{align}
Here $L_{2,\mu}$ denotes the Bochner space of (equivalence classes of) 
functions $\xi:\R^d \to \R^d$ with finite
$\|\xi \|_{L_{2,\mu}} ^2 \coloneqq \int_{\R^d} \|\xi(x)\|_2^2 \, \d \mu(x) < \infty$.
Note that  $\T_{\mu} \mathcal P_2(\R^d)$ is not a ``classical'' tangent space, in particular it is an
infinite dimensional subspace of $L_{2,\mu}$
if $\mu$ is absolutely continuous and just $\R^d$ if $\mu = \delta_x$, $x \in \R^d$.
In particular, this means that the Wasserstein space has only a ``manifold-like'' structure.

For $\lambda \in \mathbb R$, a function $\F\colon \P_2(\R^d) \to (-\infty,+\infty]$ is called 
\emph{$\lambda$-convex along geodesics} if, for every 
$\mu, \nu \in \dom \F \coloneqq \{\mu \in  \P_2(\R^d): \F(\mu) < \infty\}$,
there exists at least one geodesic $\gamma \colon [0, 1] \to \P_2(\R^d)$ 
between $\mu$ and $\nu$ such that
\begin{equation}  \label{def:lambda_convex}
    \F(\gamma(t)) 
    \le
    (1-t) \, \F(\mu) + t \, \F(\nu) 
    - \tfrac{\lambda}{2} \, t (1-t) \,  W_2^2(\mu, \nu), 
    \qquad  t \in [0,1].
\end{equation}
In the case $\lambda = 0$, we just speak about convex functions.
For a proper and lower semi-continuous (lsc) function $\F \colon \P_2(\R^d) \to (-\infty, \infty]$ 
and $\mu \in \P_2(\R^d)$, 
the \emph{reduced Fr\'echet subdifferential at $\mu$} 
is defined as 
{\small
\begin{equation}\label{need}
 \partial \F(\mu) \coloneqq \Big\{ \xi \in L_{2,\mu}:   \F(\nu) - \F(\mu)
    \ge 
    \inf_{{\tiny \pi \in \Gamma^{\opt}(\mu,\nu)}}
    \int\limits_{\R^{2d}}
    \langle \xi(x), y - x \rangle
    \, \d \pi (x, y)
    + o(W_2(\mu,\nu)) \; \forall \nu \in \P_2(\R^d) \Big\}.
\end{equation}
}%
A curve $\gamma\colon I \to \P_2(\R^d)$ 
is \emph{absolutely continuous}, if 
there exists a Borel velocity field $v_t\colon \R^d \to \R^d$ 
with $\int_I \| v_t \|_{L_{2,\gamma(t)}} \, \d t < + \infty$ such that
\begin{equation} \label{eq:CE}
  \partial_t \gamma(t) + \nabla_x \cdot ( v_t \, \gamma(t)) = 0
\end{equation}
on $I \times \R^d$ in the distributive sense, i.e., for all $\varphi \in C_{\mathrm c}^\infty(I \times \R^d)$
it holds
\begin{equation} \label{eq:CE_distr}
    \int_I \int_{\R^d} \partial_t \varphi(t, x) + v_t(x) \cdot \nabla_x \, \varphi(t, x) \, \d \gamma(t) \, \d t = 0.
\end{equation}
A locally absolutely continuous curve  
  $\gamma \colon (0,+\infty) \to \P_2(\R^d)$ 
  with velocity field $v_t \in \T_{\gamma(t)} \mathcal P_2(\R^d)$
  is called a \emph{Wasserstein gradient flow 
  with respect to} $\F\colon \P_2(\R^d) \to (-\infty, +\infty]$
  if 
  \begin{equation}\label{wgf}
      v_t  \in - \partial \F(\gamma(t)), \quad \text{for a.e. } t > 0.
  \end{equation}
	
\section{Wasserstein Gradient Flows on the Line}\label{sec:1D}
Now we restrict our attention to $d=1$, i.e., we work on the real line. 
We will see that the above notation simplifies since there is an isometric embedding 
of $\P_2(\R)$ into $L_2( (0,1) )$.
To this end, we consider 
the \emph{cumulative distribution function} 
$R_{\mu}\colon \R \to [0,1]$ of $\mu \in \P_2(\R)$,
which is defined by
$R_{\mu}(x) \coloneqq \mu((-\infty, x])$, $x \in \R$.
It is non-decreasing and right-continuous with
$\lim_{x\to -\infty} R_{\mu}(x) = 0$ 
as well as $\lim_{x\to \infty} R_{\mu}(x)=1$.
The \emph{quantile function} $Q_{\mu}\colon(0,1) \to \R$ 
is the generalized inverse of $R_\mu$ given by
\begin{equation}\label{quantile}
    Q_{\mu}(p) \coloneqq   \min \{ x \in \R \colon R_{\mu}(x) \ge p \}, \qquad p \in (0,1).
\end{equation}
It is non-decreasing and left-continuous.
The quantile functions form a convex cone 
$\mathcal C( (0,1) ) \coloneqq \{Q \in  L_2( (0,1) ): Q \text{ nondecreasing} \}$ 
in $L_2( (0,1) )$. 
Note that both the distribution and quantile functions are continuous except for at most countably many jumps. 
For a good overview see \cite[§~1.1]{RoRo14}.
By the following theorem, the mapping $\mu \mapsto Q_\mu$ is an isometric embedding 
of $\P_2(\R)$ into $L_2( (0,1) )$.

\begin{theorem}[{\cite[Thm~2.18]{Vil03}}]\label{prop:Q}
    For $\mu, \nu \in \P_2(\R)$,
    the quantile function $Q_{\mu} \in \mathcal C( (0,1) )$ 
		satisfies  $\mu = (Q_{\mu})_{\#} \lambda_{(0,1)}$
   	and 
    \begin{equation}
        W_2^2(\mu, \nu) = \int_{0}^1 |Q_{\mu}(s) - Q_{\nu}(s)|^2 \d s.
    \end{equation}
\end{theorem}

Next we will see that instead of working with functionals 
$\F: \P_2(\R) \to (-\infty,+\infty]$, we can just deal with associated
functionals ${\rm F}\colon L_2( (0,1) ) \to (-\infty,\infty]$ fulfilling
${\rm F}(Q_\mu) \coloneqq \F(\mu)$.
Note that ${\rm F}$ is defined in this way only on $\mathcal C( (0,1) )$, and there exist several continuous extensions to
the whole linear space $L_2((0,1))$.
Instead of the extended Fr\'echet subdifferential \eqref{need}, 
we will use the \emph{regular subdifferential} in $L_2((0,1))$ defined by 
\begin{equation} \label{eq:subdiff}
    \partial G(f) \coloneqq \
    \bigl\{ h \in L_2((0,1)) 
    : G(g) \ge G(f) + \langle h, g-f \rangle + o(\|g-f\|_{L_2}) \; \forall  g\in L_2((0,1))\bigr\}.
\end{equation}
The following theorem characterizes Wasserstein gradient flows by this regular subdifferential and 
states a convexity relation between $\F: \P_2(\R) \to (-\infty,+\infty]$ and the associated functional ${\rm F}$.

\begin{theorem}\label{thm:L2_representation}
{\rm i)}  Let $\gamma\colon(0,\infty)\to\P_2(\R)$ be a locally absolutely continuous curve
	and ${\rm F}\colon L_2((0,1)) \to (-\infty,\infty]$
  such that
  the pointwise derivative $\partial_t Q_{\gamma(t)}$ exists and fulfills the $L_2$ subgradient equation
      \begin{equation}
      \partial_t Q_{\gamma(t)}\in - \partial {\rm F}(Q_{\gamma(t)}),\quad\text{for almost every }t\in(0,+\infty).
    \end{equation}
 Then $\gamma$ is a Wasserstein gradient flow with respect to the functional 
$\F: \P_2(\R) \to (-\infty,+\infty]$ defined by $\F(\mu) \coloneqq {\rm F}(Q_\mu)$.
\\
{\rm ii)} If ${\rm F}: \mathcal C((0,1)) \to (-\infty,\infty]$ is convex, then $\F(\mu) \coloneqq {\rm F}(Q_\mu)$
is convex along geodesics.
\end{theorem}

\begin{proof}
i)
  Since $\gamma$ is (locally) absolute continuous,
  the velocity field $v_t$ from \eqref{eq:CE} fulfills by \cite[Prop~8.4.6]{BookAmGiSa05} for almost every $t \in (0,\infty)$ the relation
   \begin{align}
      0
      &=\lim_{h\to0}\frac{W_2(\gamma(t+h),(\mathrm{Id}+hv_t)_\#\gamma(t))}{|h|}
      \\
      &=\lim_{h\to0}\frac{W_2((Q_{\gamma(t+h)})_\#\lambda_{(0,1)},
        \bigl(Q_{\gamma(t)}+h (v_t\circ Q_{\gamma(t)}) \bigr)_\#\lambda_{(0,1)})}{|h|}\\
      &=\lim_{h\to0}\Big\|\frac{Q_{\gamma(t+h)}-Q_{\gamma(t)}}{h}-v_t\circ Q_{\gamma(t)}\Big\|_{L_2}
			=\|\partial_t Q_{\gamma(t)}-v_t\circ Q_{\gamma(t)}\|_{L_2}.
    \end{align}
  Thus, by assumption, 
  $v_t\circ Q_{\gamma(t)} \in -\partial {\rm F}(Q_{\gamma(t)})$ a.e.
  In particular, 
  for any $\mu\in\P_2(\R)$,
  we obtain
      \begin{align}
      0
      &\leq{\rm F}(Q_\mu) -{\rm F}(Q_{\gamma(t)})
      + \int_0^1 v_t(Q_{\gamma(t)}(s)) \, (Q_\mu(s)-Q_{\gamma(t)}(s)) \, \d s+o(\|Q_\mu-Q_{\gamma(t)}\|_{L_2})
      \\
      &=\mathcal F(\mu)-\mathcal F(\gamma(t))
        +\int_{\R\times\R} v_t(x) \, (y-x) \, \d \pi(x,y) +o\left(W_2(\mu,\gamma(t)) \right),
    \end{align}
   where $\pi\coloneqq (Q_{\gamma(t)},Q_\mu)_\#\lambda_{(0,1)}$.
  Since $\pi$ the unique optimal transport plan between $\gamma(t)$ and $\mu$, 
  this yields by \eqref{need} that $v_t\in -\partial \F(\gamma(t))$ showing the assertion by \eqref{wgf}.
	\\
	ii) 
	Let ${\rm F}\colon L_2((0,1))\to\R$ be convex.
	For any geodesic $\gamma: [0,1] \to \P_2(\R)$,
  since $\mu\mapsto Q_\mu$ is an isometry, 
  the curve $t\mapsto Q_{\gamma(t)}$ is a geodesic in $L_2((0,1))$ too. 
  Since $L_2((0,1))$ is a linear space,
  the convexity of ${\rm F}\colon L_2((0,1))\to\R$ yields 
  that $t\mapsto {\rm F}(Q_{\gamma(t)})={\F}(\gamma(t))$ is convex.
  Thus, $\F$ is convex along $\gamma$. \hfill \qed
\end{proof}

\begin{remark}\label{rem:convex}
If $\F\colon \P_2(\R) \to (-\infty,+\infty]$ is
proper, lsc, coercive and 
$\lambda$-convex along so-called generalized geodesics, then the
Wasserstein gradient flow starting at any $\mu_0 \in \overline{\dom \F}$
is uniquely determined and is the uniform limit of the miminizing movement scheme of
Jordan, Kinderlehrer and Otto \cite{JKO1998} when the time step size goes to zero, see
	\cite[Thm~11.2.1]{BookAmGiSa05}.
In $\mathbb R$, but not in higher dimensions, $\lambda$-convex functions along geodesics fulfill
also the stronger property that they are $\lambda$-convex along generalized geodesics,
see \cite{HGBS2022}. 
\end{remark}

\section{Discrepancies}\label{sec:discr}
We consider symmetric and \emph{conditionally positive definite kernels} 
$K\colon\R^d \times \R^d \to \R$ of order one, i.e.,
for any $n \in \N$, any pairwise different points $x^{1},\dots,x^{n} \in \R^d$ and any
$a_1, \dots, a_n \in \R$ with $\sum_{i=1}^{n} a_i = 0$
the relation
$
        \sum_{i,j=1}^n a_i a_j K(x^i, x^j) \ge 0
$
is satisfied.
Typical examples are  Riesz kernels
\begin{equation} \label{eq:riesz}
K(x,y) \coloneqq - \|x-y\|^r, \quad r \in (0,2),
\end{equation}
where we have strict inequality except for all $a_j$, $j=1,\ldots,n$ being zero.
The \emph{maximum mean discrepancy} (MMD) $\mathcal D_K^2\colon\P(\R^d) \times \P(\R^d) \to \R$ 
between two measures $\mu, \nu \in \P(\R^d)$ is defined by
    \begin{equation}         \label{eq:DK2}
        \mathcal D_K^2(\mu, \nu) \coloneqq \mathcal E_K (\mu - \nu) \quad 
    \end{equation}
    with the so-called \emph{$K$-energy} on signed measures
    \begin{equation}         \label{eq:EK}
        \mathcal E_K(\sigma) \coloneqq \frac12 \int_{\R^d} \int_{\R^d} K(x,y) \, \d \sigma(x) \d \sigma(y), \qquad \sigma \in \mathcal M(\R^d).
    \end{equation}
The relation between discrepancies and Wasserstein distances is discussed in \cite{FSVATP2018,NS2023}.
For fixed $\nu\in \P(\R^d)$, the MMD can be decomposed as
\begin{equation}
    \label{eq:dis-decomp}
    \F_\nu(\mu) = \mathcal D_K^2(\mu, \nu) 
    = \mathcal E_K(\mu) + \V_{K, \nu}(\mu) + \underbrace{\mathcal E_K(\nu)}_{\text{const.}}
\end{equation}
with the \emph{interaction energy}  on probability measures
\begin{align} \label{eq:interaction}
     \mathcal E_K(\mu) &= \frac12 \int_{\R^d} \int_{\R^d} K(x,y) \, \d \mu(x) \d \mu(y), \quad \mu \in \P_2(\R^d)
\end{align}
and the \emph{potential energy} of $\mu$
with respect to the \emph{potential} of $\nu$,
\begin{align}\label{eq:potential}
    \V_{K, \nu}(\mu) &\coloneqq \int_{\R^d} V_{K, \nu}(y) \d \mu(x), 
    \quad 
    V_{K, \nu}(x) \coloneqq - \int_{\R^d} K(x,y) \d\nu(y).       
\end{align}
In dimensions $d \ge 2$ neither $\mathcal E_K$ nor
$\mathcal D_K^2$ with the Riesz kernel are $\lambda$-convex along geodesics, see \cite{HGBS2022},
so that certain properties of Wasserstein gradient flows do not apply.
We will see that this is different on the real line.

\section{MMD Flows on the Line}\label{sec:discr_1d}
In the rest of this paper, we restrict our attention to $d=1$ and negative distance 
$K(x,y) = -|x-y|$,
i.e. to Riesz kernels with $r=1$. For fixed $\nu \in \P_2(\R)$, we consider the MMD functional
$\F_\nu\coloneqq \mathcal D_K^2(\cdot,\nu)$. Note that the unique minimizer of this functional is given by $\mu = \nu$.

\begin{lemma}\label{lem:extended_fun}
Let $\F_\nu\coloneqq \mathcal D_K^2(\cdot,\nu)$ with the negative distance kernel.
Then the convex functional ${\rm F}_\nu \colon L_2((0,1)) \to \R$  defined by
\begin{equation}
    \label{eq:disc-l2}
    {\rm F}_\nu (f) 
    \coloneqq \int_0^1 
    \Bigl((1-2s) (f(s) + Q_\nu(s)) 
    + \int_0^1 |f(s)-Q_\nu(t)| \, \d t \Bigr) \, \d s.
\end{equation}
fulfills ${\rm F}_\nu(Q_\mu)=\F_\nu(\mu)$ for all $\mu\in\P_2(\R)$. 
In particular, $\F_\nu$ is convex along (generalized) geodesics and there exists a unique
Wasserstein gradient flow.
\end{lemma}

\begin{proof}
We reformulate $\mathcal F_\nu$ as
\begin{align}     \label{eq:Fnu}
  &\F_{\nu} (\mu)
       = - \frac12 \int_{\R \times \R} |x - y| 
    (\d \mu(x) - \d \nu(x))(\d \mu(y) - \d \nu(y))
    \\
    &=
    - \frac12 \int_0^1 \int_0^1 
    |Q_\mu (s)- Q_\mu(t)| 
    - 2 |Q_\mu(s)-Q_\nu(t)| + |Q_\nu(s)-Q_\nu(t)| \, \d s \, \d t
    \\
    & = \int_0^1 \int_t^1
    Q_\mu(t)-Q_\mu(s) + Q_\nu(t)-Q_\nu(s) \, \d s \, \d t
  + \int_0^1 \int_0^1 |Q_\mu(s) - Q_\nu(t) | \, \d s \, \d t
    \\
    & = \int_0^1 \int_t^1
    Q_\mu(t) + Q_\nu(t) \, \d s \, \d t
    - \int_0^1 \int_0^s Q_\mu(s) + Q_\nu(s) \,\d t \, \d s
    + \int_0^1 \int_0^1 |Q_\mu(s) - Q_\nu(t) | \, \d s \, \d t
    \\
    &= \int_0^1 
    \Bigl((1-2s) (Q_\mu(s) + Q_\nu(s)) 
    + \int_0^1 |Q_\mu(s)-Q_\nu(t)| \, \d t \Bigr) \, \d s,
    \label{eq:Fg}
\end{align}
which yields the first claim. 
The second one follows by Theorem \ref{thm:L2_representation}ii) and Remark \ref{rem:convex}.
\hfill \qed
\end{proof}

Note that the lemma cannot immediately be generalized
to Riesz kernels with $r=(1,2)$.
\\
Finally, we derive for the special choice  $\nu=\delta_q$ in $\mathcal D_K^2(\cdot,\nu)$
an analytic formula for its Wasserstein gradient flow.

\begin{proposition}    \label{th:gradientflowdisc}
Let $\F_{\delta_q}\coloneqq \mathcal D_K^2(\cdot,\delta_q)$ with the negative distance kernel.
    Then the unique Wasserstein gradient flow of $\F_{\delta_q}$ 
    starting at $\mu_0 = \gamma(0) \in \P_2(\R)$ is 
		$\gamma(t)=(g_t)_\# \lambda_{(0,1)}$, 
    where the function $g_t\colon(0,1)\to\R$ is given by
    \begin{equation}
        \label{eq:gradflowdisc}
        g_t(s) 
        \coloneqq
        \begin{cases}
            \min\{ Q_{\mu_0}(s) + 2st, q \}, & Q_{\mu_0}(s) < q, \\
            q, & Q_{\mu_0}(s) = q,\\
            \max\{ Q_{\mu_0}(s) + 2st-2t, q \}, & Q_{\mu_0}(s) > q.
        \end{cases} 
    \end{equation}
\end{proposition}

\begin{proof}
    First, note that $g_t\in \mathcal C((0,1))$ such that it holds $g_t=Q_{\gamma(t)}$.
    Since $Q_{\delta_q} \equiv q$, 
    the subdifferential of ${\rm F}_{\delta_q}$ in \eqref{eq:disc-l2} at $g_t$
    consists of all functions
    \begin{equation}
        h(s) 
        =
       \begin{cases}
            -2s, & Q_{\mu_0}(s) < q \text{ and } t < \frac{q-Q_{\mu_0}(s)}{2s},  \\
            2-2s, & Q_{\mu_0}(s) > q \text{ and } t < \frac{Q_{\mu_0}(s)-q}{2-2s},\\
            1-2s + n(s), & \text{otherwise,}
       \end{cases}
    \end{equation}
    with $-1 \le n(s) \le 1$ for $s \in (0,1)$.
    On the other hand, the pointwise derivative of $g_t$ in \eqref{eq:gradflowdisc} can be written as
    \begin{equation}
        \partial_t g_t(s)
        =
        \begin{cases}
            2s, & Q_{\mu_0}(s) < q \text{ and } t < \frac{q-Q_{\mu_0}(s)}{2s},  \\
            2s-2, & Q_{\mu_0}(s) > q \text{ and } t < \frac{Q_{\mu_0}(s)-q}{2-2s},\\
            0,& \text{otherwise,}
        \end{cases}
    \end{equation}
    such that we obtain 
    $\partial_t Q_{\gamma(t)}=\partial_t g_t\in - \partial {\rm F}_\nu(g_t)= -\partial {\rm F}_\nu(Q_{\gamma(t)})$.
    Thus, by Lemma~\ref{lem:extended_fun} and Theorem~\ref{thm:L2_representation}, we obtain that $\gamma$ is a Wasserstein gradient flow.
    It is unique since $\F_\nu$ is convex along geodesics by Theorem~\ref{thm:L2_representation}.ii, 
		Lemma \ref{lem:extended_fun} and Remark \ref{rem:convex}.     \hfill \qed
		
		\end{proof}

\section{Intuitive Examples}\label{sec:ex}
Finally, we provide some intuitive examples of Wasserstein gradient flows of 
$\F_\nu\coloneqq \mathcal D_K^2(\cdot,\nu)$ with the negative distance kernel.

\subsection{Flow between Dirac Measures}
We consider the flow of $\F_{\delta_0}$ starting  
at the initial measure $\gamma(0) = \mu_0 \coloneqq \delta_{-1}$.
Due to $Q_{\delta_0} \equiv 0$, Proposition~\ref{th:gradientflowdisc} yields the gradient flow
$\gamma(t) \coloneqq (Q_t)_\# \lambda_{(0,1)}$ given by
\begin{equation}
    \label{eq:delta_flow}
    \gamma(t) =
    \begin{cases}
        \delta_{-1}, & t=0,\\
        \frac{1}{2t}\lambda_{[-1,-1+2t]}, & 0 \le t \le \frac{1}{2},\\
        \frac{1}{2t}\lambda_{[-1,0]} + \left(1-\frac{1}{2t}\right)\delta_0, & \frac{1}{2} < t.\\
    \end{cases}
\end{equation}
For $t \in (0,\frac12]$, the initial Dirac measure becomes a uniform measure with increasing support,
and for $t \in (\frac12,1)$ it is the convex combination of a uniform measure and $\delta_0$.
A visualization of the flow is given in Figure~\ref{fig:gradient_flow_dirac_dirac_P2R}. \hfill $\Box$

\begin{figure}
    \begin{center}
        \includegraphics[width=0.6\textwidth]{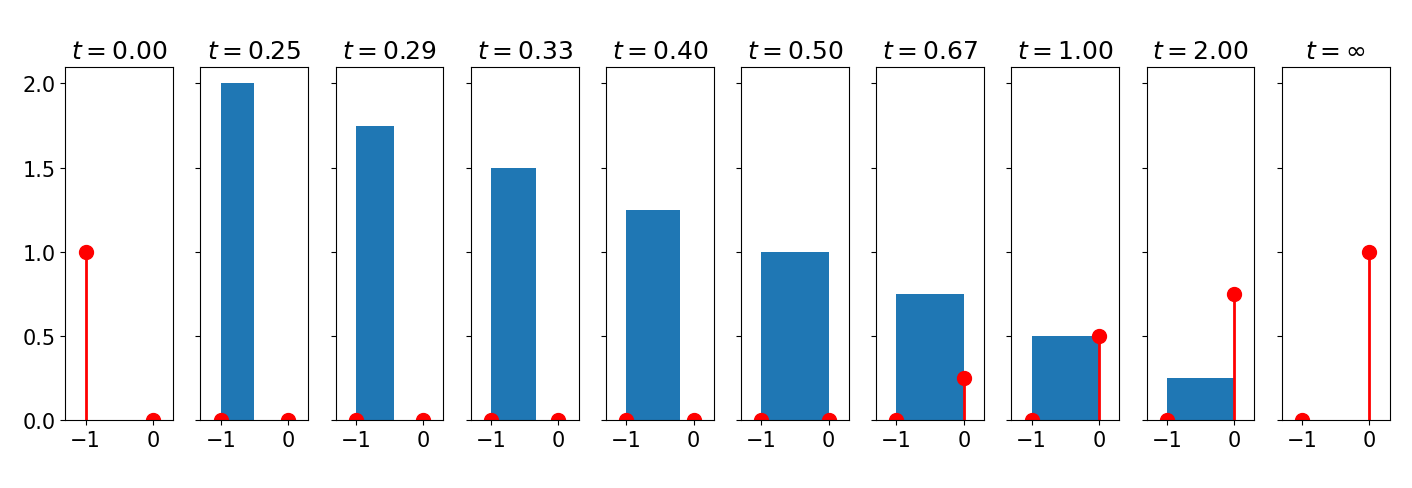}
    \end{center}
    \caption{Visualization of the Wasserstein gradient flow of $\F_{\delta_0}$
		from $\delta_{-1}$ to $\delta_0$.
    At various times $t$, the absolute continuous part is visualized by its density in blue (area equals mass) and the atomic part by the red dotted vertical line (height equals mass).
    The atomic part at the end point $x=0$ starts to grow at time $t=\tfrac12$, where the support of the density touches this point for the first time.}
    \label{fig:gradient_flow_dirac_dirac_P2R}
\end{figure}

\subsection{Flow on Restricted Sets}
Next, we are interested in the Wasserstein gradient flows
on the subsets $\mathcal S_i$, $i=1,2$, given by
\begin{enumerate}[\upshape(i)]
    \item $\mathcal S_1 \coloneqq \{ \delta_x: x \in \R\}$,
    \item 
    $\mathcal S_2 \coloneqq \{ \mu_{m, \sigma} 
    = \frac{1}{2\sqrt{3} \sigma} \lambda_{[m - \sqrt{3} \sigma, m + \sqrt{3} \sigma]}: m \in \R, \sigma \in \R_{\ge 0}\}$.
\end{enumerate}
Note that $\S_2$ is a special instance of sets of scaled and translated measures $\mu \in \P_2(\R)$ defined by
$\{ {T_{a,b}}_\# \mu: a \in \R_{\ge 0}, b \in \R\}$, 
where  $T_{a,b} (x) \coloneqq ax+b$. 
As mentioned in \cite{Ge90} the Wasserstein distance between measures $\mu_1, \mu_2$ 
from such sets has been already known to Fr\'echet:
\[
    W_2^2(\mu_1, \mu_2) = |m_{1} - m_{2}|^2 + |\sigma_{1} - \sigma_{2}|^2,
\]
where $m_{i}$ and $\sigma_{i}$ are the mean value and  standard deviation of $\mu_i$, $i=1,2$. 
This provides an isometric embedding of $\R \times \R_{\ge 0}$ into $\P_2(\R)$.
The boundary of $\mathcal S_2$ is the set of Dirac measures $\mathcal S_1$ and is isometric to $\R$.
The sets are convex in the sense that for
$\mu,\nu \in \mathcal S_i$ all geodesics $\gamma:[0,1] \to \P(\R)$ with $\gamma(0) = \mu$ and $\gamma(1)= \nu$ are in $\mathcal S_i$, $i \in \{1,2\}$.
For $i=1,2$, we consider
\begin{equation} \label{nocheinbsp}
\mathcal F_{i,\nu} (\mu)
\coloneqq 
\begin{cases}
\F_\nu & \mu \in \mathcal S_i,\\
+\infty & \rm{otherwise}.
\end{cases}
\end{equation} 
Due to the convexity of $\F_\nu$ along geodesics and the convexity of the sets $\mathcal S_i$, we obtain
that the functions $\F_{i,\nu}$ are convex along geodesics.

\subsubsection{Flows of $\F_{1,\nu}$}
We use the notation $f_x \equiv x$ for the constant function on $(0,1)$ with value $x$.
It is straightforward to check  that the function ${\rm F}\colon L^2((0,1))\to(-\infty,\infty]$ given by
\begin{align}
{\rm F}(f)=\begin{cases}F(x),&$if $f=f_x$ for some $x\in\R,\\+\infty,&$otherwise,$\end{cases}
\end{align}
with
$$
F(x)\coloneqq \int_\R|x-y|\,\d \nu(y)-\frac12\int_{\R\times\R}|y-z|\,\d \nu(y)\d \nu(z)
$$
fulfills ${\rm F}(Q_\mu)=\F_{1,\nu}(\mu)$.
In the following, we aim to find $x\colon[0,\infty)\to\R$ satisfying
    \[
        \dot x(t) = - \partial F(x(t)).
    \]
Since the set $\{Q_{\mu}:\mu\in\mathcal S_1\}$ is a one-dimensional linear subspace of $L_2((0,1))$
spanned by the constant one-function $f_1$,
this yields $f_{x(t)}\in-\partial {\rm F}(f_{x(t)})$ such that the Wasserstein gradient flow is by Theorem~\ref{thm:L2_representation} given by
$\gamma(t)=(f_{x(t)})_\#\lambda_{(0,1)}=\delta_{x(t)}$.
\\
In the special case $\nu = \delta_q$ for some $q \in \R$, we have
\[
F(x) = |x - q|, \qquad
\partial F(x) =
\begin{cases}
\{- 1\}, & x < q,\\
[-1,1], & x=q,\\
\{1\}, & x > q.
\end{cases}
\]
Therefore, the Wasserstein gradient flow for $x(0) = x_0 \ne 0$ is given by
\[
\gamma(t)=\delta_{x(t)},\quad\text{with}\quad
x(t) =
\begin{cases}
x_0 + t, & x_0<q,\\
x_0 - t, & x_0>q.
\end{cases}, \qquad 0 \le t < |x_0-q|
\]
and $\gamma(t)=\delta_q$ for $t \ge |x_0-q|$.

For $\nu = \frac12 \lambda_{[-1,1]}$ the gradient flow starting at $x_0 \in [-1,1]$ is
        \[
            x(t) = x_0 e^{-t}, \qquad t \ge 0,
        \]
        and converges to the midpoint of the interval for $t \to \infty$.
        If it starts at $x_0 \in \R \setminus [-1,1]$ the gradient flow is
        \[
            x(t) =
            \begin{cases}
                x_0 + t, & x_0<-1,\\
                x_0 - t, & x_0>1.
            \end{cases}, \qquad 0 \le t \le \min{|x_0-1|,|x_0+1|},
        \]
        where it reaches the nearest interval end point in finite time.
In Figure~\ref{fig:other}, we plotted the $x(t)$ for different initial values
$x(0)$.

\begin{figure}
    \begin{center}
        \includegraphics[width=0.25\textwidth]{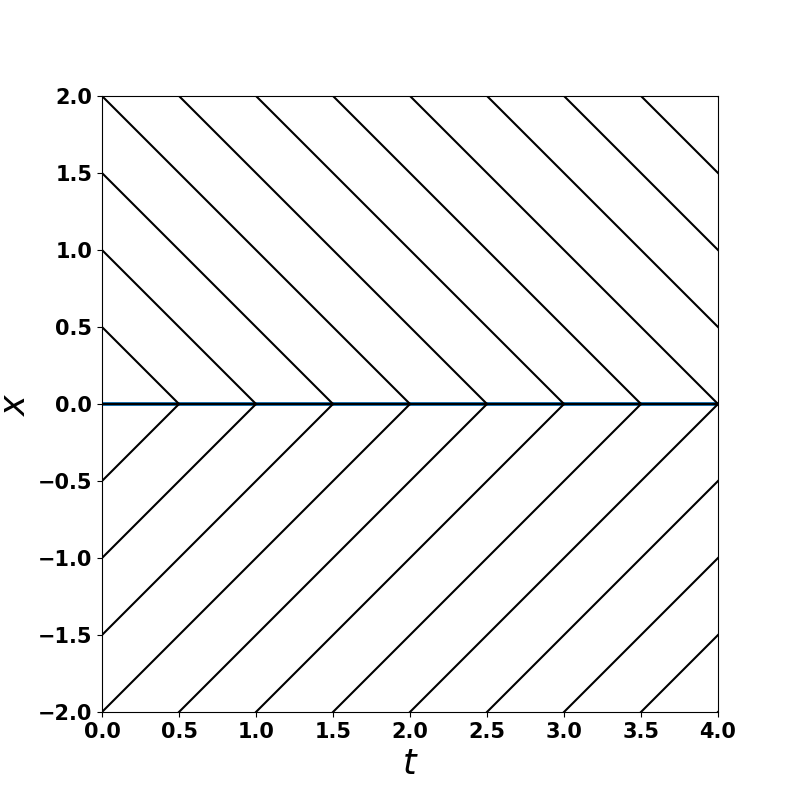}
				\hspace{1cm}
				\includegraphics[width=0.25\textwidth]{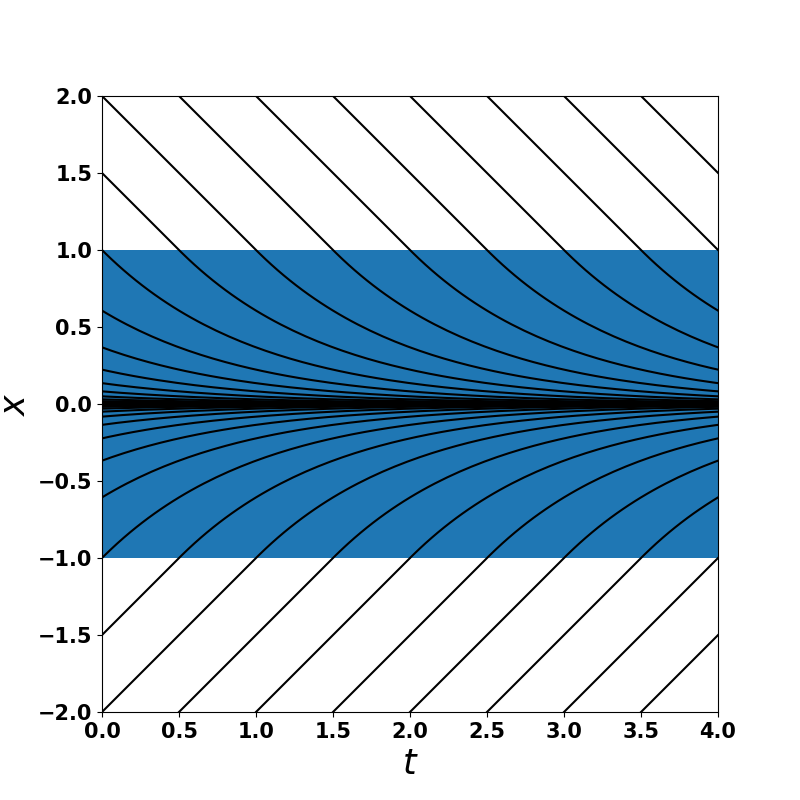}			
    \end{center}
    \caption{
		Wasserstein gradient flow of $\F_{1,\nu}$ for $\nu = \delta_0$ (left)
				and $\nu=\frac12 \lambda_{[-1,1]}$ (right) 
				from various initial points $\delta_{x}$, $x \in [-2,2]$.
				The support of the right measure $\nu$  is depicted by the blue region.
				The examples show that gradient flows may reach the optimal points in finite or infinite time.
		    }
    \label{fig:other}
\end{figure}

\subsubsection{Flows of $\F_{2,\nu}$}
We observe that $Q_{\mu_{m,\sigma}}=f_{m,\sigma}$, where $f_{m,\sigma}(x)=m+2\sqrt{3}\sigma(x-\tfrac12)$.
By Lemma~\ref{lem:extended_fun} we obtain 
that the function ${\rm F}\colon L_2((0,1))\to(-\infty,\infty]$ given by
\begin{align}
{\rm F}(f)=\begin{cases}F(m,\sigma),&$if $f=f_{m,\sigma}$ for $(m,\sigma)\in\R\times\R_{\geq0},\\+\infty,&$otherwise,$\end{cases}
\end{align}
fulfills ${\rm F}(Q_\mu)=\F_{2,\nu}(\mu)$, where
$$
F(m,\sigma)\coloneqq\int_{(0,1)}(1-2s)(f_{m,\sigma}(s)+Q_\nu(s))\d s+\int_{(0,1)^2}|f_{m,\sigma}(s)-Q_\nu(t)|\d t\d s,
$$
The set $\{f_{m,\sigma}:m,\sigma\in\R\}$ is a two dimensional linear subspace of $L_2((0,1))$ with orthonormal basis $\{f_{1,0},f_{0,1}\}$.
We aim to compute $m\colon[0,\infty)\to\R$ and $\sigma\colon[0,\infty)\to\R_{\geq0}$ with
    \begin{equation}\label{eq:ode_in_R2}
        (\dot m(t), \dot \sigma(t)) = - \partial F(m(t), \sigma(t)), \qquad t \in I \subset \R,
    \end{equation}
because this yields $f_{m(t),\sigma(t)}\in-\partial {\rm F}(f_{m(t),\sigma(t)})$ such that
$\gamma(t)=(f_{m(t),\sigma(t)})_\#\lambda_{(0,1)}=\mu_{m,\sigma}$ is by Theorem~\ref{thm:L2_representation} the Wasserstein gradient flow.

In the following, we consider the special case $\nu=\delta_0=\mu_{0,0}$. Then, the function $F$ reduces to
\begin{align}
F(m,\sigma)&=\int_\R(1-2s)(m+2\sqrt{3}\sigma(s-\tfrac12))+|m+2\sqrt{3}\sigma(s-\tfrac12)|\d s\\[-2pt]
&=-\frac{\sigma}{\sqrt{3}}+\begin{cases}|m|,&$if $|m|\geq\sqrt{3}\sigma,\\
\frac{m^2+3\sigma^2}{2\sqrt{3}\sigma^2}&$if $ |m|<\sqrt{3}\sigma,\end{cases}
\end{align}
and the subdifferential is given by
$$
\partial F(m,\sigma)=\begin{cases}\mathrm{sgn}(m)\times\{-\tfrac1{\sqrt{3}}\},
&$if $|m|\geq\sqrt{3}\sigma,\\
\{(\tfrac{m}{\sqrt{3}\sigma^2},\tfrac{-m^2}{\sqrt{3}\sigma^3} - \frac{1}{\sqrt{3}})\},
&$if $|m| < \sqrt{3}\sigma,\end{cases}
\quad \mathrm{sgn}(m)=\begin{cases}\{\tfrac{|m|}{m}\},&$if $m \ne 0,\\ [-1,1],&$if $m=0.\end{cases}
$$
We observe that $F$ is differentiable for $\sigma>0$. Thus, for any initial intial value $(m(0),\sigma(0))=(m_0,\sigma_0)$, we
can compute the trajectory $(m(t),\sigma(t))$ solving \eqref{eq:ode_in_R2} using an ODE solver.
In Figure~\ref{fig:other_1} (left), we plotted the level sets of the function $F(m,\sigma)$ as well
as the solution trajectory $(m(t),\sigma(t))$ for different initial values $(m(0),\sigma(0))$.
For $(m(0),\sigma(0))=(-1,0)$, the resulting flow is illustrated in Figure~\ref{fig:other_1}, right.

\begin{figure}    \begin{center}
        \includegraphics[width=0.35\textwidth]{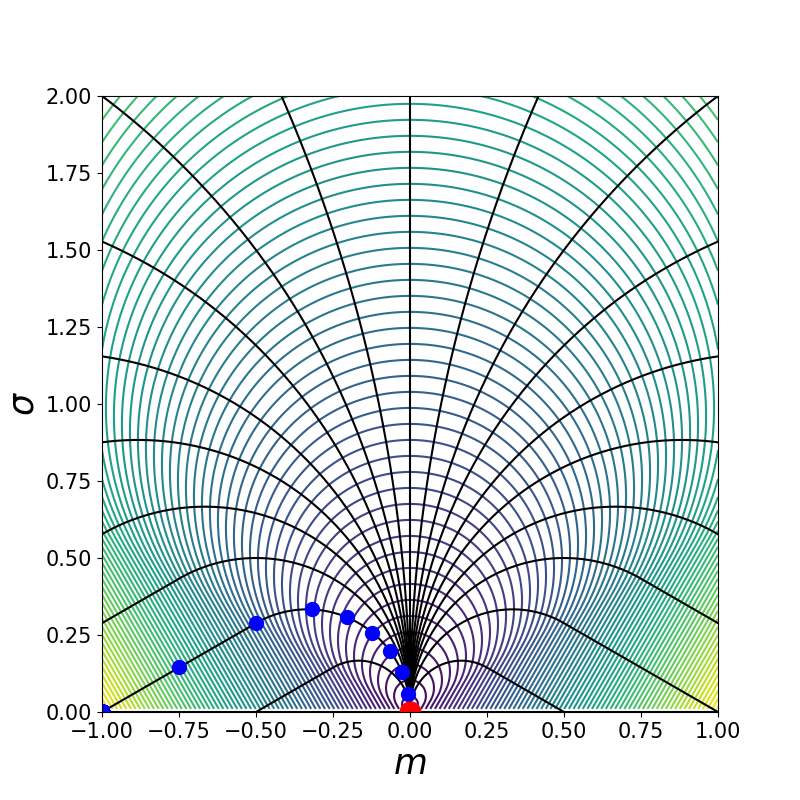} 
											\hspace{0.5cm}
				\includegraphics[width=0.6\textwidth]{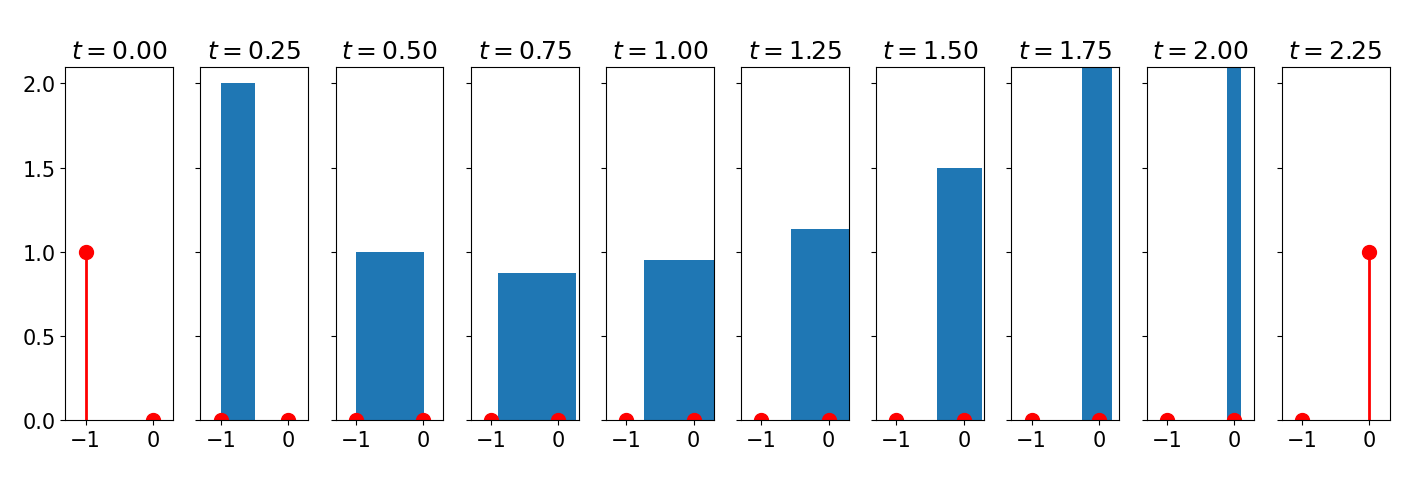}
    \end{center}
    \caption{Wasserstein gradient flow  $\F_{2,\delta_0}$ from $(m(0),\sigma(0))$ to $\delta_0$ (left) 
		and from $\delta_{-1}$ to $\delta_0$ (right). 
		In contrast Figure \ref{fig:gradient_flow_dirac_dirac_P2R} it is a uniform measure for all $t \in (0,1)$.
    }
    \label{fig:other_1}
\end{figure}

\paragraph{Flows for a Smooth Kernel}
  For smooth, positive definite kernels $K$ the MMD functional $\F_{\nu} \coloneqq \mathcal D_K^2(\cdot, \nu)$ is in general not convex and leads to a more complex energy landscape than for the negative distance kernel. 
This may lead to problems for optimization algorithms.
  To illustrate this observation, we let $\nu \coloneqq \lambda_{[-1,1]}$ and compare the energy landscape of the restricted functional 
	$\F_{2,\nu}$ for $K(x,y) \coloneqq -|x-y|$ and the kernel 
\begin{equation}
\label{eq:rbf_kernel}
     \tilde K(x,y) \coloneqq
            \begin{cases}
            (1-\tfrac12 |x-y|)^2 (|x-y| + 1), & |x-y| \le 2,\\
                0, & \text{else.}
            \end{cases}
\end{equation}
  In contrast to the negative distance kernel $K$, the kernel $\tilde K$ is positive definite (without restrictions on the $a_i$), cf. \cite{wendland2005}, and has a Lipschitz continuous gradient.
  The two energy landscapes of $\F_{2,\nu}$ are visualized in Figure~\ref{fig:energylandscape_comparison}. The non-convexity of $\F_{\nu}$ for $\tilde K$ is readily seen by the presence of a saddle point for $\F_{2,\nu}$ at $\mu = \delta_0$ (equivalently to $(m,\sigma) = (0,0)$ in the $m\sigma$-plane).
Note that any Wasserstein gradient flow of $\F_{\nu}$ starting at a Dirac measure $\delta_x$ converges to this saddle point $\mu = \delta_0$.

  \begin{figure}
    \begin{center}
        \includegraphics[width=0.35\textwidth]{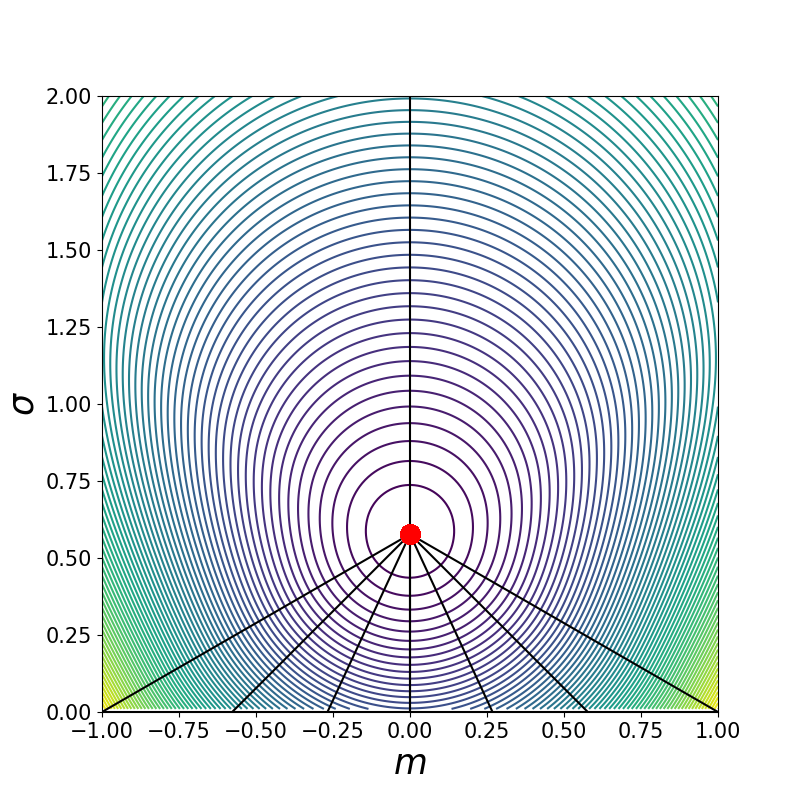}
        \hspace{1cm}
        \includegraphics[width=0.35\textwidth]{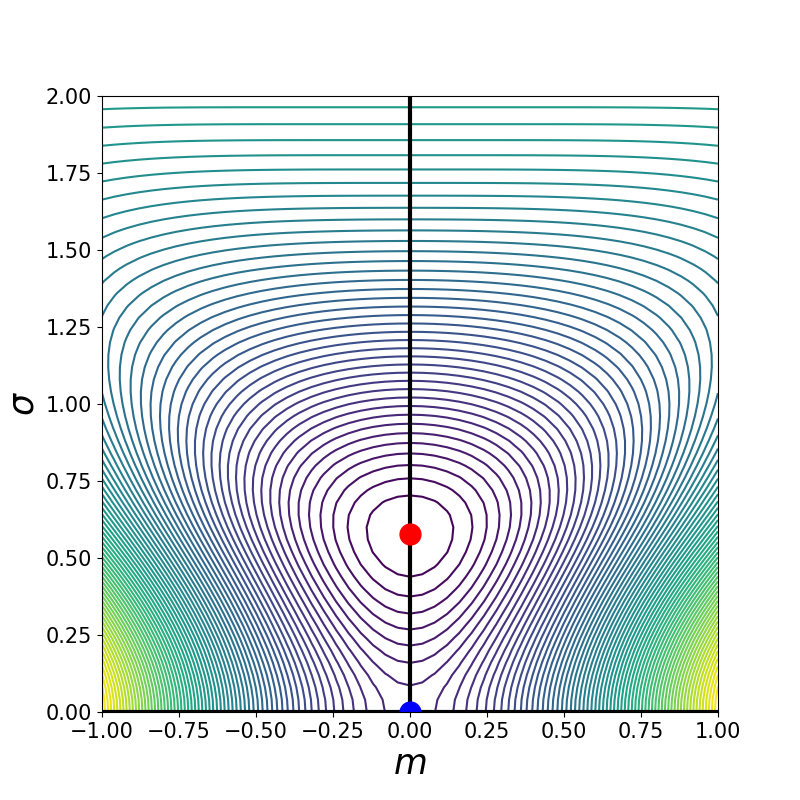}
    \end{center}
    \caption{Visualization of the energy landscapes of $\F_{2,\lambda_{[-1,1]}}$ for the convex negative distance kernel (left) and the non-convex, smooth kernel given in \eqref{eq:rbf_kernel}. 
		The red dot is the global minimizer $\lambda_{[-1,1]}$ (left and right) and the blue point (right) is the saddle point $\delta_0$. The black lines depict selected gradient flows.
    }
    \label{fig:energylandscape_comparison}
\end{figure}

\section{Conclusions} \label{sec:conc}
We provided insight into Wasserstein gradient flows of MMD functionals with negative distance kernels
and characterized in particular flows ending in a Dirac measure.
We have seen that such flows are not simple particle flows, e.g. starting in another Dirac measure the flow becomes immediately
uniformly distributed and after a certain time a mixture of a uniform and a Dirac measure.
In our future work, we want to extend our considerations to empirical measures and incorporate deep learning techniques
as in \cite{AHS2023}.
Also the treatment of other functionals which incorporate an interaction energy part appears to be interesting.
Further, we may combine univariate techniques with multivariate settings using Radon transform like techniques 
as in \cite{BCSD2022,LSMDS2019,NHPB2021}.

\bibliographystyle{splncs04}
\bibliography{references}
\end{document}